\documentclass{amsart}
\usepackage{amssymb}
\usepackage{amsmath}

\newtheorem{theorem}{Theorem}[section]
\newtheorem{lemma}[theorem]{Lemma}
\theoremstyle{definition}

\newtheorem{example}[theorem]{Example}

\theoremstyle{remark}
\newtheorem{remark}[theorem]{Remark}
\numberwithin{equation}{section}

\input{tcilatex}

\begin{document}
\title{Dirac Systems Associated with Hahn Difference Operator}
\author{Fatma H\i ra}
\address{Hitit University, Arts and Science Faculty, Department of
Mathematics,  \c{C}orum, Turkey, 19040}
\email{fatmahira@yahoo.com.tr, fatmahira@hitit.edu.tr}
\date{}
\dedicatory{}

\begin{abstract}
In this paper, we introduce $q,\omega -$Dirac system. We investigate the
existence and uniqueness of solutions for this system and obtain some
spectral properties based on the Hahn difference operator. Also we give two
examples, which indicate that asymptotic formulas for eigenvalues.
\end{abstract}

\maketitle







\section{Introduction and Preliminaries}

In $\left[ 1,2\right] ,$ Hahn introduced the difference operator $%
D_{q,\omega }$ which is defined by%
\begin{equation}
\left\{ 
\begin{array}{c}
D_{q,\omega }f\left( t\right) :=\dfrac{f\left( qt+\omega \right) -f\left(
t\right) }{\left( qt+\omega \right) -t},~t\neq \omega _{0}, \\ 
f^{\prime }\left( \omega _{0}\right) ,\text{ \ \ \ \ \ \ \ \ \ \ \ \ \ \ \ \
\ \ \ \ \ \ \ }t=\omega _{0},%
\end{array}%
\right.  \tag{1.1}
\end{equation}%
where $q\in \left( 0,1\right) ,\omega >0$ are fixed and $\omega _{0}:=\omega
/\left( 1-q\right) .$ This operator extends the forward difference operator%
\begin{equation}
\Delta _{\omega }f\left( t\right) :=\dfrac{f\left( t+\omega \right) -f\left(
t\right) }{\left( t+\omega \right) -t},  \tag{1.2}
\end{equation}%
where $\omega >0$ is fixed $\left( \text{see [3-6]}\right) $ as well as
Jackson $q-$difference operator%
\begin{equation}
D_{q}f\left( t\right) :=\frac{f\left( qt\right) -f\left( t\right) }{t\left(
q-1\right) }  \tag{1.3}
\end{equation}%
where $q\in \left( 0,1\right) $ is fixed $\left( \text{see [7-12]}\right) .$

In $\left[ 13\right] ,$ the authors gave a rigorous analysis of Hahn's
difference operator and the associated calculus. The existence and
uniqueness theorems for general first-order $q,\omega -$initial value
problems and the theory of linear Hahn difference equations were studied in $%
\left[ 14\right] $ and $\left[ 15\right] ,$ respectively. Recently, a $%
q,\omega -$Sturm-Liouville theory has been established in $\left[ 16\right] $
and sampling theorems associated with $q,\omega -$Sturm-Liouville problems
in the regular setting have been derived in $\left[ 17\right] .$

In $\left[ 18\right] ,$ the authors presented the $q-$analog of the one
dimensional Dirac system: 
\begin{equation}
\left\{ 
\begin{array}{l}
-\dfrac{1}{q}D_{q^{-1}}y_{2}+p\left( x\right) y_{1}=\lambda y_{1}, \\ 
D_{q}y_{1}+r\left( x\right) y_{2}=\lambda y_{2},%
\end{array}%
\right.   \tag{1.4}
\end{equation}%
\begin{equation}
k_{11}y_{1}\left( 0\right) +k_{12}y_{2}\left( 0\right) =0,  \tag{1.5}
\end{equation}%
\begin{equation}
k_{21}y_{1}\left( a\right) +k_{22}y_{2}\left( aq^{-1}\right) =0,  \tag{1.6}
\end{equation}%
where $k_{ij}~\left( i,j=1,2\right) $ are real numbers, $y\left( x\right)
=\left( 
\begin{array}{c}
y_{1}\left( x\right)  \\ 
y_{2}\left( x\right) 
\end{array}%
\right) ,$ $0\leq x\leq a<\infty .$ The existence and uniqueness of the
solution, some spectral properties and asymptotic formulas for the
eigenvalues and the eigenfunctions of this $q-$Dirac system were
investigated in $[18]$ and $[19].$

In this paper, we introduce a $q,\omega -$version of $q-$Dirac system
(1.4)-(1.6). If the $q-$difference operator $D_{q}$ is replaced by the Hahn
difference operator $D_{q,\omega },$ then we obtain the following $q,\omega -
$Dirac system. Namely, we obtain the system which consists of the $q,\omega -
$Dirac equations%
\begin{equation}
\left\{ 
\begin{array}{l}
-\dfrac{1}{q}D_{\frac{1}{q},\frac{-\omega }{q}}y_{2}+p\left( t\right)
y_{1}=\lambda y_{1}, \\ 
D_{q,\omega }y_{1}+r\left( t\right) y_{2}=\lambda y_{2},%
\end{array}%
\right.   \tag{1.7}
\end{equation}%
and the boundary conditions

\begin{equation}
B_{1}\left( y\right) :=k_{11}y_{1}\left( \omega _{0}\right)
+k_{12}y_{2}\left( \omega _{0}\right) =0,  \tag{1.8}
\end{equation}%
\begin{equation}
B_{2}\left( y\right) :=k_{21}y_{1}\left( a\right) +k_{22}y_{2}\left(
h^{-1}\left( a\right) \right) =0,  \tag{1.9}
\end{equation}%
where $\omega _{0}\leq t\leq a<\infty ,$ $k_{ij}~\left( i,j=1,2\right) $ are
real numbers, $\lambda \in 
\mathbb{C}
,$ $p\left( .\right) $ and $r\left( .\right) $ are real-valued functions
defined on $\left[ \omega _{0},a\right] $ and continuous at $\omega _{0}$,
and $h\left( t\right) $ is a function defined below. We will establish an
existence and uniqueness of the solution of $q,\omega -$Dirac equation
(1.7). Also we will discuss some spectral properties of the eigenvalues and
the eigenfunctions of this system (1.7)-(1.9). Finally, we will give two
examples, which indicate that asymptotic formulas for the eigenvalues.

Let $h\left( t\right) :=qt+\omega ,~t\in I;$ an interval of $%
\mathbb{R}
$ containing $\omega _{0}.$ One can see that $k$th order iteration of $%
h\left( t\right) $ is given by%
\begin{equation*}
h^{k}\left( t\right) =q^{k}t+\omega \left[ k\right] _{q},~t\in I.
\end{equation*}%
The sequence $h^{k}\left( t\right) $ is uniformly convergent to $\omega _{0}$
on $I.$ Here $\left[ k\right] _{q}$ is defined by%
\begin{equation*}
\left[ k\right] _{q}=\frac{1-q^{k}}{1-q}.
\end{equation*}

The following relation directly follows from the definition $D_{q,\omega }$%
\begin{equation}
\left( D_{q,\omega }f\right) \left( h^{-1}\left( t\right) \right) =D_{\dfrac{%
1}{q},\dfrac{-\omega }{q}}f\left( t\right) ,  \tag{1.10}
\end{equation}%
where $h^{-1}\left( t\right) :=\left( t-\omega \right) /q,~t\in I.$ The $%
q,\omega -$type product formula is given by%
\begin{equation}
D_{q,\omega }\left( fg\right) \left( t\right) =D_{q,\omega }\left( f\left(
t\right) \right) g\left( t\right) +f\left( qt+\omega \right) D_{q,\omega
}g\left( t\right) .  \tag{1.11}
\end{equation}%
The $q,\omega -$integral is introduced in $\left[ 13\right] $ to be the
Jackson-N\"{o}rlund sum%
\begin{equation}
\int\limits_{a}^{b}f\left( t\right) d_{q,\omega }t=\int\limits_{\omega
_{0}}^{b}f\left( t\right) d_{q,\omega }t-\int\limits_{\omega
_{0}}^{a}f\left( t\right) d_{q,\omega }t,  \tag{1.12}
\end{equation}%
where $\omega _{0}<a<b,~a,b\in I,$ and 
\begin{equation}
\int\limits_{\omega _{0}}^{x}f\left( t\right) d_{q,\omega }t=\left( x\left(
1-q\right) -\omega \right) \sum\limits_{k=0}^{\infty }q^{k}f\left(
xq^{k}+\omega \left[ k\right] _{q}\right) ,~x\in I,  \tag{1.13}
\end{equation}%
provided that the series converges. The fundamental theorem of $q,\omega -$%
calculus given in $\left[ 13\right] ~$states that if $f:I\rightarrow 
\mathbb{R}
$ is continuous at $\omega _{0},$ and%
\begin{equation}
F\left( t\right) :=\int\limits_{\omega _{0}}^{t}f\left( x\right) d_{q,\omega
}x,~x\in I,  \tag{1.14}
\end{equation}%
then $F$ is continuous at $\omega _{0}.$ Furthermore, $D_{q,\omega }F\left(
t\right) $ exists for every $t\in I$ and%
\begin{equation}
D_{q,\omega }F\left( t\right) =f\left( t\right) .  \tag{1.15}
\end{equation}%
Conversely, 
\begin{equation}
\int\limits_{a}^{b}D_{q,\omega }f\left( t\right) d_{q,\omega }t=f\left(
b\right) -f\left( a\right) ,~\text{for all }a,b\in I.  \tag{1.16}
\end{equation}

The $q,\omega -$integration by parts for continuous functions $f,~g$ is
given in $\left[ 13\right] $ by%
\begin{equation}
\int\limits_{a}^{b}f\left( t\right) D_{q,\omega }g\left( t\right)
d_{q,\omega }t=f\left( t\right) g\left( t\right) \left\vert _{a}^{b}\right.
-\int\limits_{a}^{b}D_{q,\omega }\left( f\left( t\right) \right) g\left(
qt+\omega \right) d_{q,\omega }t,~a,b\in I.  \tag{1.17}
\end{equation}

Trigonometric functions of $q,\omega -$cosine and sine are defined by

\begin{equation}
C_{q,\omega }\left( t,\mu \right) :=\sum\limits_{n=0}^{\infty }\frac{\left(
-1\right) ^{n}q^{n^{2}}\left( \mu \left( t\left( 1-q\right) -\omega \right)
\right) ^{2n}}{\left( q;q\right) _{2n}},\text{ }t\in 
\mathbb{C}
,  \tag{1.18}
\end{equation}%
\begin{equation}
S_{q,\omega }\left( t,\mu \right) :=\sum\limits_{n=0}^{\infty }\frac{\left(
-1\right) ^{n}q^{n\left( n+1\right) }\left( \mu \left( t\left( 1-q\right)
-\omega \right) \right) ^{2n+1}}{\left( q;q\right) _{2n+1}},\text{ }t\in 
\mathbb{C}
.  \tag{1.19}
\end{equation}%
Here $\left( q;q\right) _{k}~$is the $q-$shifted factorial%
\begin{equation}
\left( q;q\right) _{k}:=\left\{ 
\begin{array}{l}
1,\text{ \ \ \ \ \ }k=0, \\ 
\dprod\limits_{j=1}^{k}\left( 1-q^{j}\right) ,~k=1,2,...~.%
\end{array}%
\right.  \tag{1.20}
\end{equation}%
Furthermore, these functions satisfy%
\begin{equation}
D_{q,\omega }S_{q,\omega }\left( t,\mu \right) =\mu \,C_{q,\omega }\left(
t,q^{\frac{1}{2}}\mu \right) ,~\text{\ \ \ \ }D_{q,\omega }C_{q,\omega
}\left( t,\mu \right) =-q^{\frac{1}{2}}\mu S_{q,\omega }\left( t,q^{\frac{1}{%
2}}\mu \right) .  \tag{1.21}
\end{equation}

It is not hard to see that $C_{q,\omega }\left( .\right) $ and $S_{q,\omega
}\left( .\right) $ satisfy the initial value problems%
\begin{equation}
-\frac{1}{q}D_{\frac{1}{q},\frac{-\omega }{q}}D_{q,\omega }y\left( t\right)
=y\left( t\right) ,~t\in 
\mathbb{R}
,  \tag{1.22}
\end{equation}%
subject to the initial conditions%
\begin{equation}
y\left( \omega _{0}\right) =1,~D_{q,\omega }y\left( \omega _{0}\right) =0;%
\text{ \ \ }y\left( \omega _{0}\right) =0,~D_{q,\omega }y\left( \omega
_{0}\right) =1,  \tag{1.23}
\end{equation}%
respectively. If we let $t\in 
\mathbb{C}
,$ then it can be proved that $C_{q,\omega }\left( t\right) $ and $%
S_{q,\omega }\left( t\right) $ are entire functions of order zero.

Let $a>0$ be fixed and $L_{q,\omega }^{2}\left( \omega _{0},a\right) $ be
the set of all complex valued functions defined on $\left[ \omega _{0},a%
\right] $ for which%
\begin{equation*}
\left\Vert f\left( .\right) \right\Vert =\left( \left( \int\limits_{\omega
_{0}}^{a}\left\vert f\left( t\right) \right\vert ^{2}d_{q,\omega }t\right)
^{1\backslash 2}<\infty \right) .
\end{equation*}%
The space $L_{q,\omega }^{2}\left( \omega _{0},a\right) $ is a separable
Hilbert space with inner product%
\begin{equation}
\left\langle f,g\right\rangle :=\int\limits_{\omega _{0}}^{a}f\left(
t\right) \overline{g\left( t\right) }d_{q,\omega }t,~\ f,g\in L_{q,\omega
}^{2}\left( \omega _{0},a\right) ,  \tag{1.24}
\end{equation}%
where $\overline{z}$ denotes the complex conjugate of $z\in 
\mathbb{C}
$ $\left( \text{see }\left[ 16\right] \right) .$ Let $C_{q,\omega
}^{2}\left( \omega _{0},a\right) $\ be the subspace of \ $L_{q,\omega
}^{2}\left( \omega _{0},a\right) $, which consists of all functions $y\left(
.\right) $\ for which $y\left( .\right) $\ , $D_{q,\omega }y\left( .\right) $%
\ are continuous at $\omega _{0}$. Let $H_{q,\omega }$\ be the Hilbert space%
\begin{equation*}
H_{q,\omega }:=\left\{ y\left( .\right) =\left( 
\begin{array}{c}
y_{1}\left( .\right) \\ 
y_{2}\left( .\right)%
\end{array}%
\right) ,~y_{1},y_{2}\in C_{q,\omega }^{2}\left( \omega _{0},a\right)
\right\} .
\end{equation*}%
The inner product of $H_{q,\omega }$\ is defined by 
\begin{equation}
\left\langle y\left( .\right) ,~z\left( .\right) \right\rangle _{H_{q,\omega
}}:=\int\limits_{\omega _{0}}^{a}y^{\intercal }\left( t\right) z\left(
t\right) d_{q,\omega }t,  \tag{1.25}
\end{equation}%
where $\intercal $\ denotes the matrix transpose.

\section{Fundamental Solutions and Spectral Properties}

In this section, we give an existence and uniqueness theorem of the $%
q,\omega -$system (1.7).

Let $y\left( .\right) =\left( 
\begin{array}{c}
y_{1}\left( .\right)  \\ 
y_{2}\left( .\right) 
\end{array}%
\right) ,~z\left( .\right) =\left( 
\begin{array}{c}
z_{1}\left( .\right)  \\ 
z_{2}\left( .\right) 
\end{array}%
\right) \in H_{q,\omega }.$ Then the $q,\omega -$Wronskian of $y\left(
.\right) $ and $z\left( .\right) $ is defined by%
\begin{equation}
W_{q,\omega }\left( y,z\right) \left( t\right) :=y_{1}\left( t\right)
z_{2}\left( h^{-1}\left( t\right) \right) -z_{1}\left( t\right) y_{2}\left(
h^{-1}\left( t\right) \right) .  \tag{2.1}
\end{equation}

\begin{lemma}
The $q,\omega -$Wronskian of solutions of the $q,\omega -$system (1.7) is
independent of both $t$ and $\lambda .$
\end{lemma}

\begin{proof}
Let $y\left( t,\lambda \right) =\left( 
\begin{array}{c}
y_{1}\left( t,\lambda \right)  \\ 
y_{2}\left( t,\lambda \right) 
\end{array}%
\right) $ and $z\left( t,\lambda \right) =\left( 
\begin{array}{c}
z_{1}\left( t,\lambda \right)  \\ 
z_{2}\left( t,\lambda \right) 
\end{array}%
\right) $ be\ two solutions of the $q,\omega -$system (1.7). From (2.1) and
(1.11), we have%
\begin{equation}
\begin{array}{l}
D_{q,\omega }W_{q,\omega }\left( y,z\right) \left( t,\lambda \right)
=D\left( _{q,\omega }z_{2}\left( h^{-1}\left( t\right) ,\lambda \right)
\right) y_{1}\left( t,\lambda \right) +z_{2}\left( t,\lambda \right)
D_{q,\omega }y_{1}\left( t,\lambda \right)  \\ 
~\text{\ \ \ \ \ \ \ \ \ \ \ \ \ \ \ \ \ \ \ \ \ \ \ \ \ \ }-D_{q,\omega
}\left( y_{2}\left( h^{-1}\left( t\right) ,\lambda \right) \right)
z_{1}\left( t,\lambda \right) -y_{2}\left( t,\lambda \right) D_{q,\omega
}z_{1}\left( t,\lambda \right) .%
\end{array}
\tag{2.2}
\end{equation}%
Since $y\left( t,\lambda \right) $ and $z\left( t,\lambda \right) $ are
solutions of the $q,\omega -$system (1.7) and by using (1.10) we get%
\begin{equation}
D_{q,\omega }W_{q,\omega }\left( y,z\right) \left( t,\lambda \right) =0. 
\tag{2.3}
\end{equation}%
Therefore, $W_{q,\omega }\left( y,z\right) \left( t,\lambda \right) $ is a
constant.
\end{proof}

\begin{corollary}
If $y\left( t,\lambda \right) $ and $z\left( t,\lambda \right) $ are both
solutions of the $q,\omega -$system (1.7), then either $W_{q,\omega }\left(
y,z\right) \left( t,\lambda \right) =0$\ or $W_{q,\omega }\left( y,z\right)
\left( t,\lambda \right) \neq 0$ \ for all $t\in \left[ \omega _{0},a\right] $%
.
\end{corollary}

\begin{theorem}
For $\lambda \in 
\mathbb{C}
$, the $q,\omega -$system (1.7) has a unique solution $\phi \left( t,\lambda
\right) =\left( 
\begin{array}{c}
\phi _{1}\left( t,\lambda \right) \\ 
\phi _{2}\left( t,\lambda \right)%
\end{array}%
\right) $\ subject to the initial conditions%
\begin{equation*}
\phi _{1}\left( \omega _{0},\lambda \right) =c_{1},\text{ \ }\phi _{2}\left(
\omega _{0},\lambda \right) =c_{2},
\end{equation*}
where $c_{1},c_{2}\in 
\mathbb{C}
.$
\end{theorem}

\begin{proof}
Let%
\begin{equation}
\QATOP{\varphi _{1}\left( t,\lambda \right) =\left( 
\begin{array}{c}
\varphi _{11}\left( t,\lambda \right) \\ 
\varphi _{12}\left( t,\lambda \right)%
\end{array}%
\right) =\left( 
\begin{array}{c}
C_{q,\omega }\left( t,\lambda \right) \\ 
-\sqrt{q}S_{q,\omega }\left( t,\lambda \sqrt{q}\right)%
\end{array}%
\right) ,}{\varphi _{2}\left( t,\lambda \right) =\left( 
\begin{array}{c}
\varphi _{21}\left( t,\lambda \right) \\ 
\varphi _{22}\left( t,\lambda \right)%
\end{array}%
\right) =\left( 
\begin{array}{c}
S_{q,\omega }\left( t,\lambda \right) \\ 
C_{q,\omega }\left( t,\lambda \sqrt{q}\right)%
\end{array}%
\right) .}  \tag{2.5}
\end{equation}

Then the function 
\begin{equation}
y\left( t,\lambda \right) =\left( 
\begin{array}{c}
y_{1}\left( t,\lambda \right) \\ 
y_{2}\left( t,\lambda \right)%
\end{array}%
\right) =\left( 
\begin{array}{c}
c_{1}\varphi _{11}\left( t,\lambda \right) +c_{2}\varphi _{21}\left(
t,\lambda \right) \\ 
c_{1}\varphi _{12}\left( t,\lambda \right) +c_{2}\varphi _{22}\left(
t,\lambda \right)%
\end{array}%
\right) ,  \tag{2.6}
\end{equation}%
is a fundamental set of the $q,\omega -$system (1.7) for $p\left( t\right)
=r\left( t\right) =0.$ It is not hard to see that $W_{q,\omega }\left(
\varphi _{1},\varphi _{1}\right) \left( t,\lambda \right) =1.$

Define the sequence $\left\{ \psi _{m}\left( .,\lambda \right) \right\}
_{m=1}^{\infty }=\left( 
\begin{array}{c}
\psi _{m1}\left( .,\lambda \right) \\ 
\psi _{m2}\left( .,\lambda \right)%
\end{array}%
\right) _{m=1}^{\infty }$ of successive approximations by%
\begin{equation}
\psi _{1}\left( t,\lambda \right) =\left( 
\begin{array}{c}
\psi _{11}\left( t,\lambda \right) \\ 
\psi _{12}\left( t,\lambda \right)%
\end{array}%
\right) =\left( 
\begin{array}{c}
c_{1}\varphi _{11}\left( t,\lambda \right) +c_{2}\varphi _{21}\left(
t,\lambda \right) \\ 
c_{1}\varphi _{12}\left( t,\lambda \right) +c_{2}\varphi _{22}\left(
t,\lambda \right)%
\end{array}%
\right) ,  \tag{2.7}
\end{equation}

\begin{equation}
\begin{array}{l}
\psi _{m+1}\left( t,\lambda \right) =\left( 
\begin{array}{c}
\psi _{\left( m+1\right) 1}\left( t,\lambda \right) \\ 
\psi _{\left( m+1\right) 2}\left( t,\lambda \right)%
\end{array}%
\right) \\ 
=\left( 
\begin{array}{l}
{\small \psi }_{11}\left( t,\lambda \right) {\small +q}\dint\limits_{\omega
_{0}}^{t}\left\{ \varphi _{21}\left( t,\lambda \right) \varphi _{11}\left(
h\left( s\right) ,\lambda \right) -\varphi _{11}\left( t,\lambda \right)
\varphi _{21}\left( h\left( s\right) ,\lambda \right) \right\} {\small p}%
\left( h\left( s\right) \right) {\small \psi }_{m1}\left( h\left( s\right)
,\lambda \right) {\small d}_{q,\omega }{\small s} \\ 
{\small +}\dint\limits_{\omega _{0}}^{t}\left\{ \varphi _{21}\left(
t,\lambda \right) \varphi _{12}\left( s,\lambda \right) -\varphi _{11}\left(
t,\lambda \right) \varphi _{22}\left( s,\lambda \right) \right\} {\small r}%
\left( s\right) {\small \psi }_{m2}\left( s,\lambda \right) {\small d}%
_{q,\omega }{\small s} \\ 
{\small \psi }_{12}\left( t,\lambda \right) {\small +q}\dint\limits_{\omega
_{0}}^{t}\left\{ \varphi _{22}\left( t,\lambda \right) \varphi _{11}\left(
h\left( s\right) ,\lambda \right) -\varphi _{12}\left( t,\lambda \right)
\varphi _{21}\left( h\left( s\right) ,\lambda \right) \right\} {\small p}%
\left( h\left( s\right) \right) {\small \psi }_{m1}\left( h\left( s\right)
,\lambda \right) {\small d}_{q,\omega }{\small s} \\ 
{\small +}\dint\limits_{\omega _{0}}^{t}\left\{ \varphi _{22}\left(
t,\lambda \right) \varphi _{12}\left( s,\lambda \right) -\varphi _{12}\left(
t,\lambda \right) \varphi _{22}\left( s,\lambda \right) \right\} {\small r}%
\left( s\right) {\small \psi }_{m2}\left( s,\lambda \right) {\small d}%
_{q,\omega }{\small s}%
\end{array}%
\right)%
\end{array}
\tag{2.8}
\end{equation}%
$m=1,2,3,...~.$ For a fixed $\lambda \in 
\mathbb{C}
,$ there exist positive numbers $B\left( \lambda \right) ,$ $A_{1},~A_{2}$
and $A$ independent of $t$ such that%
\begin{equation}
\begin{array}{l}
\left\vert p\left( t\right) \right\vert \leq A_{1},~\left\vert r\left(
t\right) \right\vert \leq A_{2},\text{ }A=\max \left\{ A_{1},A_{2}\right\} ,
\\ 
\left\vert \varphi _{ij}\left( t,\lambda \right) \right\vert \leq \sqrt{%
\frac{B\left( \lambda \right) }{2}}~i,j=1,2,\text{ }t\in \left( \omega
_{0},a\right) .%
\end{array}
\tag{2.9}
\end{equation}

Let $K\left( \lambda \right) :=\left( \left\vert c_{1}\right\vert
+\left\vert c_{1}\right\vert \right) \sqrt{\frac{B\left( \lambda \right) }{2}%
}.$ Then from (2.9), $\left\vert \psi _{11}\left( t,\lambda \right)
\right\vert \leq K\left( \lambda \right) $ and $\left\vert \psi _{12}\left(
t,\lambda \right) \right\vert \leq K\left( \lambda \right) .$ Using
mathematical induction, we have 
\begin{equation}
\begin{array}{l}
\left\vert \psi _{\left( m+1\right) 1}\left( t,\lambda \right) -\psi
_{m1}\left( t,\lambda \right) \right\vert \leq K\left( \lambda \right) \frac{%
1}{2}\left( -1;q\right) _{m+1}\frac{\left( AB\left( \lambda \right) \left(
t\left( 1-q\right) -\omega \right) \right) ^{m}}{\left( q;q\right) _{m}}, \\ 
\left\vert \psi _{\left( m+1\right) 2}\left( t,\lambda \right) -\psi
_{m2}\left( t,\lambda \right) \right\vert \leq K\left( \lambda \right) \frac{%
1}{2}\left( -1;q\right) _{m+1}\frac{\left( AB\left( \lambda \right) \left(
t\left( 1-q\right) -\omega \right) \right) ^{m}}{\left( q;q\right) _{m}},%
\end{array}%
m=1,2,...  \tag{2.10}
\end{equation}

To guarantee the converge of the term on the right side in (2.10), we assume
additionally that%
\begin{equation}
\left\vert t-\omega _{0}\right\vert <\frac{1}{\left\vert AB\left( \lambda
\right) \left( 1-q\right) \right\vert },~t\in \left( \omega _{0},a\right) . 
\tag{2.11}
\end{equation}

Consequently by Weierstrass' test the series 
\begin{equation}
\begin{array}{l}
\psi _{11}\left( t,\lambda \right) +\dsum\limits_{m=1}^{\infty }\psi
_{\left( m+1\right) 1}\left( t,\lambda \right) -\psi _{m1}\left( t,\lambda
\right)  \\ 
\psi _{12}\left( t,\lambda \right) +\dsum\limits_{m=1}^{\infty }\psi
_{\left( m+1\right) 2}\left( t,\lambda \right) -\psi _{m2}\left( t,\lambda
\right) 
\end{array}
\tag{2.12}
\end{equation}%
converges uniformly on $\left( \omega _{0},a\right) .$ Since the $m$th
partial sums of the series is $\psi _{m+1}\left( t,\lambda \right) =\left( 
\begin{array}{c}
\psi _{\left( m+1\right) 1}\left( t,\lambda \right)  \\ 
\psi _{\left( m+1\right) 2}\left( t,\lambda \right) 
\end{array}%
\right) ,$ then $\psi _{m+1}\left( .,\lambda \right) $ approaches a function 
$\phi \left( .,\lambda \right) =\left( 
\begin{array}{c}
\phi _{1}\left( .,\lambda \right)  \\ 
\phi _{2}\left( .,\lambda \right) 
\end{array}%
\right) $ uniformly on $\left( \omega _{0},a\right) .$ We can also prove by
induction on $m$ that $\psi _{mi}\left( t,\lambda \right) $ and $D_{q,\omega
}\psi _{mi}\left( t,\lambda \right) $ $\left( i=1,2\right) $ are continuous
at $\omega _{0}$, where%
\begin{equation}
\begin{array}{l}
D_{q,\omega }\psi _{\left( m+1\right) 1}\left( t,\lambda \right)
=c_{1}D_{q,\omega }\varphi _{11}\left( t,\lambda \right) +c_{2}D_{q,\omega
}\varphi _{21}\left( t,\lambda \right)  \\ 
+{\small q}\dint\limits_{\omega _{0}}^{t}\left\{ D_{q,\omega }\left( \varphi
_{21}\left( t,\lambda \right) \right) \varphi _{11}\left( h\left( s\right)
,\lambda \right) -D_{q,\omega }\left( \varphi _{11}\left( t,\lambda \right)
\right) \varphi _{21}\left( h\left( s\right) ,\lambda \right) \right\} 
{\small p}\left( h\left( s\right) \right) {\small \psi }_{m1}\left( h\left(
s\right) ,\lambda \right) {\small d}_{q,\omega }{\small s} \\ 
+\dint\limits_{\omega _{0}}^{t}\left\{ D_{q,\omega }\left( \varphi
_{21}\left( t,\lambda \right) \right) \varphi _{12}\left( s,\lambda \right)
-D_{q,\omega }\left( \varphi _{11}\left( t,\lambda \right) \right) \varphi
_{22}\left( s,\lambda \right) \right\} {\small r}\left( s\right) {\small %
\psi }_{m2}\left( s,\lambda \right) {\small d}_{q,\omega }{\small s}%
\end{array}
\tag{2.13}
\end{equation}%
\begin{equation}
\begin{array}{l}
D_{q,\omega }\psi _{\left( m+1\right) 2}\left( t,\lambda \right)
=c_{1}D_{q,\omega }\varphi _{12}\left( t,\lambda \right) +c_{2}D_{q,\omega
}\varphi _{22}\left( t,\lambda \right)  \\ 
+{\small q}\dint\limits_{\omega _{0}}^{t}\left\{ D_{q,\omega }\left( \varphi
_{22}\left( t,\lambda \right) \right) \varphi _{11}\left( h\left( s\right)
,\lambda \right) -D_{q,\omega }\left( \varphi _{12}\left( t,\lambda \right)
\right) \varphi _{21}\left( h\left( s\right) ,\lambda \right) \right\} 
{\small p}\left( h\left( s\right) \right) {\small \psi }_{m1}\left( h\left(
s\right) ,\lambda \right) {\small d}_{q,\omega }{\small s} \\ 
+\dint\limits_{\omega _{0}}^{t}\left\{ D_{q,\omega }\left( \varphi
_{22}\left( t,\lambda \right) \right) \varphi _{12}\left( s,\lambda \right)
-D_{q,\omega }\left( \varphi _{12}\left( t,\lambda \right) \right) \varphi
_{22}\left( s,\lambda \right) \right\} {\small r}\left( s\right) {\small %
\psi }_{m2}\left( s,\lambda \right) {\small d}_{q,\omega }{\small s}%
\end{array}
\tag{2.14}
\end{equation}%
$m=1,2,3,...~.$ Therefore, the functions $\phi _{i}\left( .,\lambda \right) $%
\ and $D_{q,\omega }\phi _{i}\left( .,\lambda \right) $\ ~$\left(
i=1,2\right) ~$are continuous at $\omega _{0}$, i.e., $\phi \left( .,\lambda
\right) \in H_{q,\omega }.~$Let $m\rightarrow \infty $ in (2.8), we obtain $%
\phi \left( t,\lambda \right) =\left( 
\begin{array}{c}
\phi _{1}\left( t,\lambda \right)  \\ 
\phi _{2}\left( t,\lambda \right) 
\end{array}%
\right) $ where%
\begin{equation}
\begin{array}{l}
\phi _{1}\left( t,\lambda \right) =c_{1}\varphi _{11}\left( t,\lambda
\right) +c_{2}\varphi _{21}\left( t,\lambda \right)  \\ 
{\small +q}\dint\limits_{\omega _{0}}^{t}\left\{ \varphi _{21}\left(
t,\lambda \right) \varphi _{11}\left( h\left( s\right) ,\lambda \right)
-\varphi _{11}\left( t,\lambda \right) \varphi _{21}\left( h\left( s\right)
,\lambda \right) \right\} {\small p}\left( h\left( s\right) \right) \phi
_{1}\left( h\left( s\right) ,\lambda \right) {\small d}_{q,\omega }{\small s}
\\ 
{\small +}\dint\limits_{\omega _{0}}^{t}\left\{ \varphi _{21}\left(
t,\lambda \right) \varphi _{12}\left( s,\lambda \right) -\varphi _{11}\left(
t,\lambda \right) \varphi _{22}\left( s,\lambda \right) \right\} {\small r}%
\left( s\right) \phi _{2}\left( s,\lambda \right) {\small d}_{q,\omega }%
{\small s,}%
\end{array}
\tag{2.15}
\end{equation}%
\begin{equation}
\begin{array}{l}
\phi _{2}\left( t,\lambda \right) =c_{1}\varphi _{12}\left( t,\lambda
\right) +c_{2}\varphi _{22}\left( t,\lambda \right)  \\ 
{\small +q}\dint\limits_{\omega _{0}}^{t}\left\{ \varphi _{22}\left(
t,\lambda \right) \varphi _{11}\left( h\left( s\right) ,\lambda \right)
-\varphi _{12}\left( t,\lambda \right) \varphi _{21}\left( h\left( s\right)
,\lambda \right) \right\} {\small p}\left( h\left( s\right) \right) \phi
_{1}\left( h\left( s\right) ,\lambda \right) {\small d}_{q,\omega }{\small s}
\\ 
{\small +}\dint\limits_{\omega _{0}}^{t}\left\{ \varphi _{22}\left(
t,\lambda \right) \varphi _{12}\left( s,\lambda \right) -\varphi _{12}\left(
t,\lambda \right) \varphi _{22}\left( s,\lambda \right) \right\} {\small r}%
\left( s\right) \phi _{2}\left( s,\lambda \right) {\small d}_{q,\omega }%
{\small s.}%
\end{array}
\tag{2.16}
\end{equation}

Now we prove that $\phi \left( .,\lambda \right) $ solves the $q,\omega -$%
system (1.7) subject to (2.4). It is not hard to see that $\phi _{1}\left(
\omega _{0},\lambda \right) =c_{1}$ and $\phi _{2}\left( \omega _{0},\lambda
\right) =c_{2}.$ From (2.15) and using (1.11) and (1.15), we have 
\begin{equation}
\begin{array}{l}
D_{q,\omega }\phi _{1}\left( t,\lambda \right) =c_{1}D_{q,\omega }\varphi
_{11}\left( t,\lambda \right) +c_{2}D_{q,\omega }\varphi _{21}\left(
t,\lambda \right) \\ 
+{\small q}\dint\limits_{\omega _{0}}^{t}\left\{ D_{q,\omega }\left( \varphi
_{21}\left( t,\lambda \right) \right) \varphi _{11}\left( h\left( s\right)
,\lambda \right) -D_{q,\omega }\left( \varphi _{11}\left( t,\lambda \right)
\right) \varphi _{21}\left( h\left( s\right) ,\lambda \right) \right\} 
{\small p}\left( h\left( s\right) \right) \phi _{1}\left( h\left( s\right)
,\lambda \right) {\small d}_{q,\omega }{\small s} \\ 
+\dint\limits_{\omega _{0}}^{t}\left\{ D_{q,\omega }\left( \varphi
_{21}\left( t,\lambda \right) \right) \varphi _{12}\left( s,\lambda \right)
-D_{q,\omega }\left( \varphi _{11}\left( t,\lambda \right) \right) \varphi
_{22}\left( s,\lambda \right) \right\} {\small r}\left( s\right) \phi
_{2}\left( s,\lambda \right) {\small d}_{q,\omega }{\small s} \\ 
+\left\{ \varphi _{21}\left( h\left( t\right) ,\lambda \right) \varphi
_{12}\left( t,\lambda \right) -\varphi _{11}\left( h\left( t\right) ,\lambda
\right) \varphi _{22}\left( t,\lambda \right) \right\} {\small r}\left(
t\right) \phi _{2}\left( t,\lambda \right) .%
\end{array}
\tag{2.17}
\end{equation}%
Since the function (2.6) is a fundamental solution set of the $q,\omega -$%
system (1.7) for $p\left( t\right) =0=r\left( t\right) $\ and from Corollary
2.2, we get%
\begin{equation}
D_{q,\omega }\phi _{1}\left( t,\lambda \right) =\lambda \phi _{2}\left(
t,\lambda \right) -r\left( t\right) \phi _{2}\left( t,\lambda \right) . 
\tag{2.18}
\end{equation}%
The validity of the other equation in the $q,\omega -$system (1.7) is proved
similarly. For the uniqueness assume $Y\left( t,\lambda \right) =\left( 
\begin{array}{c}
Y_{1}\left( t,\lambda \right) \\ 
Y_{2}\left( t,\lambda \right)%
\end{array}%
\right) $ and $Z\left( t,\lambda \right) =\left( 
\begin{array}{c}
Z_{1}\left( t,\lambda \right) \\ 
Z_{2}\left( t,\lambda \right)%
\end{array}%
\right) $ are two solutions of the $q,\omega -$system (1.7) together with
(2.4). Applying the $q,\omega -$\ integration to the $q,\omega -$system
(1.7), yields%
\begin{equation}
Y_{1}\left( t,\lambda \right) =c_{1}+\dint\limits_{\omega _{0}}^{t}\left(
\lambda -{\small r}\left( s\right) \right) Y_{2}\left( s,\lambda \right) 
{\small d}_{q,\omega }{\small s,}  \tag{2.19}
\end{equation}%
\begin{equation}
Y_{2}\left( t,\lambda \right) =c_{2}+q\dint\limits_{\omega _{0}}^{t}\left(
p\left( h\left( s\right) \right) -\lambda \right) Y_{1}\left( h\left(
s\right) ,\lambda \right) {\small d}_{q,\omega }{\small s,}  \tag{2.20}
\end{equation}%
and%
\begin{equation}
Z_{1}\left( t,\lambda \right) =c_{1}+\dint\limits_{\omega _{0}}^{t}\left(
\lambda -{\small r}\left( s\right) \right) Z_{2}\left( s,\lambda \right) 
{\small d}_{q,\omega }{\small s,}  \tag{2.21}
\end{equation}%
\begin{equation}
Z_{2}\left( t,\lambda \right) =c_{2}+q\dint\limits_{\omega _{0}}^{t}\left(
p\left( h\left( s\right) \right) -\lambda \right) Z_{1}\left( h\left(
s\right) ,\lambda \right) {\small d}_{q,\omega }{\small s.}  \tag{2.22}
\end{equation}%
Since $Y\left( t,\lambda \right) ,~Z\left( t,\lambda \right) ,~p\left(
t\right) $ and $r\left( t\right) $ are continuous at $\omega _{0}$, then
exist positive numbers $N_{\lambda ,t}$, $M_{\lambda ,t}$ \ such that%
\begin{equation}
\begin{array}{l}
\sup_{n\in 
\mathbb{N}
}\left\vert Y_{i}\left( h^{n}\left( t\right) ,\lambda \right) \right\vert
=N_{i\lambda ,t},~~\sup_{n\in 
\mathbb{N}
}\left\vert Z_{i}\left( h^{n}\left( t\right) ,\lambda \right) \right\vert =%
\widetilde{N}_{i\lambda ,t},~ \\ 
N_{\lambda ,t}=\max \left\{ N_{i\lambda ,t},\widetilde{\text{ }N}_{i\lambda
,t}\right\} ,~i=1,2,%
\end{array}
\tag{2.23}
\end{equation}%
\begin{equation}
\begin{array}{l}
\sup_{n\in 
\mathbb{N}
}\left\vert \lambda -r\left( h^{n}\left( t\right) \right) \right\vert
=M_{1\lambda ,t},~\sup_{n\in 
\mathbb{N}
}\left\vert \lambda -p\left( h^{n}\left( t\right) \right) \right\vert
=M_{2\lambda ,t}, \\ 
M_{\lambda ,t}=\max \left\{ M_{1\lambda ,t},~M_{2\lambda ,t}\right\} .%
\end{array}
\tag{2.24}
\end{equation}

We can prove by mathematical induction on $k$ that 
\begin{equation}
\left\vert Y_{1}\left( t,\lambda \right) -Z_{1}\left( t,\lambda \right)
\right\vert \leq 2q^{\left\lfloor \frac{k^{2}}{4}\right\rfloor }M_{\lambda
,t}^{k}N_{\lambda ,t}\frac{\left( t\left( 1-q\right) -\omega \right) ^{k}}{%
\left( q;q\right) _{k}},  \tag{2.25}
\end{equation}%
\begin{equation}
\left\vert Y_{2}\left( t,\lambda \right) -Z_{2}\left( t,\lambda \right)
\right\vert \leq 2q^{\frac{1}{8}\left\{ 2k^{2}+4k+\left( -1\right)
^{k+1}+1\right\} }M_{\lambda ,t}^{k}N_{\lambda ,t}\frac{\left( t\left(
1-q\right) -\omega \right) ^{k}}{\left( q;q\right) _{k}},  \tag{2.26}
\end{equation}%
$k\in 
\mathbb{N}
,~t\in \left( \omega _{0},a\right) .~$Indeed, if (2.25) holds at $k\in 
\mathbb{N}
$, then from (2.19) and (2.21)%
\begin{equation}
\begin{array}{l}
\left\vert Y_{1}\left( t,\lambda \right) -Z_{1}\left( t,\lambda \right)
\right\vert \leq M_{\lambda ,t}2q^{\frac{1}{8}\left\{ 2k^{2}+4k+\left(
-1\right) ^{k+1}+1\right\} }\frac{M_{\lambda ,t}^{k}N_{\lambda ,t}}{\left(
q;q\right) _{k}}\dint\limits_{\omega _{0}}^{t}\left( s\left( 1-q\right)
-\omega \right) ^{k}{\small d}_{q,\omega }{\small s} \\ 
\text{ \ \ \ \ \ \ \ \ \ \ \ \ \ \ \ \ }=2q^{\frac{1}{8}\left\{
2k^{2}+4k+\left( -1\right) ^{k+1}+1\right\} }\frac{M_{\lambda
,t}^{k+1}N_{\lambda ,t}}{\left( q;q\right) _{k}}\sum\limits_{n=1}^{\infty
}\left( t\left( 1-q\right) -\omega \right) q^{n}q^{nk}\left( t\left(
1-q\right) -\omega \right) ^{k} \\ 
\text{ \ \ \ \ \ \ \ \ \ \ \ \ \ \ \ \ }=2q^{\frac{1}{8}\left\{
2k^{2}+4k+\left( -1\right) ^{k+1}+1\right\} }N_{\lambda ,t}\frac{\left(
M_{\lambda ,t}\left( t\left( 1-q\right) -\omega \right) \right) ^{k+1}}{%
\left( q;q\right) _{k+1}}.%
\end{array}
\tag{2.27}
\end{equation}%
Similarly, if (2.26) holds at $k\in 
\mathbb{N}
$, then from (2.20) and (2.22)%
\begin{equation}
\begin{array}{l}
\left\vert Y_{2}\left( t,\lambda \right) -Z_{2}\left( t,\lambda \right)
\right\vert \leq qM_{\lambda ,t}2q^{\left\lfloor \frac{k^{2}}{4}%
\right\rfloor }\frac{M_{\lambda ,t}^{k}N_{\lambda ,t}}{\left( q;q\right) _{k}%
}\dint\limits_{\omega _{0}}^{t}\left( h\left( s\right) \left( 1-q\right)
-\omega \right) ^{k}{\small d}_{q,\omega }{\small s} \\ 
\text{ \ \ \ \ \ \ \ \ \ \ \ \ }=2qq^{\left\lfloor \frac{k^{2}}{4}%
\right\rfloor }\frac{M_{\lambda ,t}^{k+1}N_{\lambda ,t}}{\left( q;q\right)
_{k}}\sum\limits_{n=1}^{\infty }\left( t\left( 1-q\right) -\omega \right)
q^{n}q^{\left( n+1\right) k}\left( t\left( 1-q\right) -\omega \right) ^{k}
\\ 
\text{ \ \ \ \ \ \ \ \ \ \ \ \ }=2q^{k+1+\left\lfloor \frac{k^{2}}{4}%
\right\rfloor }N_{\lambda ,t}\frac{\left( M_{\lambda ,t}\left( t\left(
1-q\right) -\omega \right) \right) ^{k+1}}{\left( q;q\right) _{k+1}}.%
\end{array}
\tag{2.28}
\end{equation}%
Hence (2.25) and (2.26) hold true at $\left( k+1\right) .$ Consequently
(2.25) and (2.26) are true for all $k\in 
\mathbb{N}
$\ because from (2.23) it is satisfied at $k=0.$Since%
\begin{equation}
\lim_{k\rightarrow \infty }2q^{\left\lfloor \frac{k^{2}}{4}\right\rfloor
}M_{\lambda ,t}^{k}N_{\lambda ,t}\frac{\left( t\left( 1-q\right) -\omega
\right) ^{k}}{\left( q;q\right) _{k}}=0,  \tag{2.29}
\end{equation}%
and%
\begin{equation}
\lim_{k\rightarrow \infty }2q^{\frac{1}{8}\left\{ 2k^{2}+4k+\left( -1\right)
^{k+1}+1\right\} }M_{\lambda ,t}^{k}N_{\lambda ,t}\frac{\left( t\left(
1-q\right) -\omega \right) ^{k}}{\left( q;q\right) _{k}}=0,  \tag{2.30}
\end{equation}%
then $Y_{1}\left( t,\lambda \right) =Z_{1}\left( t,\lambda \right) $ and $%
Y_{2}\left( t,\lambda \right) =Z_{2}\left( t,\lambda \right) $ for al $t\in
\left( \omega _{0},a\right) $, i.e., \ $Y\left( t,\lambda \right) =Z\left(
t,\lambda \right) $. This proves the uniqueness.
\end{proof}

\begin{lemma}
If $\lambda _{1}$\ and $\lambda _{2}$\ are two different eigenvalues of the $%
q,\omega -$system (1.7)-(1.9), then the corresponding eigenfunctions $%
y\left( t,\lambda _{1}\right) $\ and $z\left( t,\lambda _{2}\right) $ are
orthogonal, i.e., 
\begin{equation}
\dint\limits_{\omega _{0}}^{a}y^{\intercal }\left( t,\lambda _{1}\right)
z\left( t,\lambda _{2}\right) {\small d}_{q,\omega }{\small t=}%
\dint\limits_{\omega _{0}}^{a}\left\{ y_{1}\left( t,\lambda _{1}\right)
z_{1}\left( t,\lambda _{2}\right) +y_{2}\left( t,\lambda _{1}\right)
z_{2}\left( t,\lambda _{2}\right) \right\} {\small d}_{q,\omega }{\small t=0,%
}  \tag{2.31}
\end{equation}%
where $y^{\intercal }=\left( y_{1},y_{2}\right) .$
\end{lemma}

\begin{proof}
Since $y\left( t,\lambda _{1}\right) $\ and $z\left( t,\lambda _{2}\right) $
are solutions of the $q,\omega -$system (1.7)-(1.9)%
\begin{equation*}
\left\{ 
\begin{array}{l}
-\dfrac{1}{q}D_{\frac{1}{q},\frac{-\omega }{q}}y_{2}\left( t,\lambda
_{1}\right) +\left\{ p\left( t\right) -\lambda _{1}\right\} y_{1}\left(
t,\lambda _{1}\right) =0, \\ 
D_{q,\omega }y_{1}\left( t,\lambda _{1}\right) +\left\{ r\left( t\right)
-\lambda _{1}\right\} y_{2}\left( t,\lambda _{1}\right) =0,%
\end{array}%
\right.
\end{equation*}%
and%
\begin{equation*}
\left\{ 
\begin{array}{l}
-\dfrac{1}{q}D_{\frac{1}{q},\frac{-\omega }{q}}z_{2}\left( t,\lambda
_{2}\right) +\left\{ p\left( t\right) -\lambda _{2}\right\} z_{1}\left(
t,\lambda _{2}\right) =0, \\ 
D_{q,\omega }z_{1}\left( t,\lambda _{2}\right) +\left\{ r\left( t\right)
-\lambda _{2}\right\} z_{2}\left( t,\lambda _{2}\right) =0.%
\end{array}%
\right.
\end{equation*}%
Multiplying by $z_{1}$,$~z_{2},-y_{1}~$and $-y_{2},~$respectively, and
adding together we have%
\begin{equation}
\begin{array}{l}
-\dfrac{1}{q}D_{\frac{1}{q},\frac{-\omega }{q}}\left( y_{2}\left( t,\lambda
_{1}\right) \right) z_{1}\left( t,\lambda _{2}\right) +D_{q,\omega }\left(
y_{1}\left( t,\lambda _{1}\right) \right) z_{2}\left( t,\lambda _{2}\right)
\\ 
+\dfrac{1}{q}D_{\frac{1}{q},\frac{-\omega }{q}}\left( z_{2}\left( t,\lambda
_{2}\right) \right) y_{1}\left( t,\lambda _{1}\right) -D_{q,\omega }\left(
z_{1}\left( t,\lambda _{2}\right) \right) y_{2}\left( t,\lambda _{1}\right)
\\ 
=\lambda _{1}\left\{ y_{1}\left( t,\lambda _{1}\right) z_{1}\left( t,\lambda
_{2}\right) +y_{2}\left( t,\lambda _{1}\right) z_{2}\left( t,\lambda
_{2}\right) \right\} -\lambda _{2}\left\{ y_{1}\left( t,\lambda _{1}\right)
z_{1}\left( t,\lambda _{2}\right) +y_{2}\left( t,\lambda _{1}\right)
z_{2}\left( t,\lambda _{2}\right) \right\} .%
\end{array}
\tag{2.32}
\end{equation}%
Using (1.10) and (1.11), we obtain%
\begin{equation}
\begin{array}{l}
D_{q,\omega }\left\{ y_{1}\left( t,\lambda _{1}\right) z_{2}\left(
h^{-1}\left( t\right) ,\lambda _{2}\right) -y_{2}\left( h^{-1}\left(
t\right) ,\lambda _{1}\right) z_{1}\left( t,\lambda _{2}\right) \right\} \\ 
=\left( \lambda _{1}-\lambda _{2}\right) \left\{ y_{1}\left( t,\lambda
_{1}\right) z_{1}\left( t,\lambda _{2}\right) +y_{2}\left( t,\lambda
_{1}\right) z_{2}\left( t,\lambda _{2}\right) \right\} .%
\end{array}
\tag{2.33}
\end{equation}%
Applying the $q,\omega -$integration to (2.33), yields%
\begin{equation}
\begin{array}{l}
\left( \lambda _{1}-\lambda _{2}\right) \dint\limits_{\omega
_{0}}^{a}\left\{ y_{1}\left( t,\lambda _{1}\right) z_{1}\left( t,\lambda
_{2}\right) +y_{2}\left( t,\lambda _{1}\right) z_{2}\left( t,\lambda
_{2}\right) \right\} {\small d}_{q,\omega }{\small t} \\ 
\left. \left\{ y_{1}\left( t,\lambda _{1}\right) z_{2}\left( h^{-1}\left(
t\right) ,\lambda _{2}\right) -y_{2}\left( h^{-1}\left( t\right) ,\lambda
_{1}\right) z_{1}\left( t,\lambda _{2}\right) \right\} \right\vert _{\omega
_{0}}^{a}.%
\end{array}
\tag{2.34}
\end{equation}%
It follows from the boundary conditions (1.8) and (1.9) the right-hand side
vanishes. It is concluded that%
\begin{equation}
\left( \lambda _{1}-\lambda _{2}\right) \dint\limits_{\omega
_{0}}^{a}y^{\intercal }\left( t,\lambda _{1}\right) z\left( t,\lambda
_{2}\right) {\small d}_{q,\omega }{\small t=0.}  \tag{2.35}
\end{equation}%
The lemma is thus proved, since $\lambda _{1}\neq \lambda _{2}.$
\end{proof}

\begin{lemma}
The eigenvalues of the $q,\omega -$system (1.7)-(1.9) are real.
\end{lemma}

\begin{proof}
Assume the contrary that $\lambda _{0}$ is a nonreal eigenvalue of the $%
q,\omega -$system (1.7)-(1.9). Let $y\left( t,\lambda _{0}\right) $ be a
corresponding (nontrivial) eigenfunction. $\overline{\lambda }_{0}$ is also
an eigenvalue, correponding to the egenfunction $\overline{y}\left( t,%
\overline{\lambda }_{0}\right) .$ Since $\lambda _{0}\neq $ $\overline{%
\lambda }_{0}$ by the previous lemma%
\begin{equation}
\dint\limits_{\omega _{0}}^{a}\left\{ \left\vert y_{1}\left( t,\lambda
_{0}\right) \right\vert ^{2}+\left\vert y_{2}\left( t,\lambda _{0}\right)
\right\vert ^{2}\right\} {\small d}_{q,\omega }{\small t=0.}  \tag{2.36}
\end{equation}%
Hence $y\left( t,\lambda _{0}\right) =0$\ and this is a contradiction.
Consequently, $\lambda _{0}$ must be real.
\end{proof}

\begin{lemma}
The eigenvalues of the $q,\omega -$system (1.7)-(1.9) are simple.
\end{lemma}

\begin{proof}
Let $\phi \left( t,\lambda \right) =\left( 
\begin{array}{c}
\phi _{1}\left( t,\lambda \right) \\ 
\phi _{2}\left( t,\lambda \right)%
\end{array}%
\right) $ be a solution of the $q,\omega -$system (1.7) together with 
\begin{equation}
\phi _{1}\left( \omega _{0},\lambda \right) =k_{12},~~\ \phi _{2}\left(
\omega _{0},\lambda \right) =-k_{11}.  \tag{2.37}
\end{equation}%
It is obvious that $\phi \left( t,\lambda \right) $ satisfies the boundary
condition (1.8). To find the eigenvalues of the $q,\omega -$system
(1.7)-(1.9) \ we have to insert this function into the boundary condition
(1.9) and find the roots of the obtained equation. So, putting the function $%
\phi \left( t,\lambda \right) $ into the boundary condition (1.9) we get the
following characteristic function%
\begin{equation}
\Delta \left( \lambda \right) =k_{21}\phi _{1}\left( a,\lambda \right)
+k_{22}\phi _{2}\left( h^{-1}\left( a\right) ,\lambda \right) .  \tag{2.38}
\end{equation}%
Then $\frac{d\Delta \left( \lambda \right) }{d\lambda }=k_{21}\frac{\partial
\phi _{1}\left( a,\lambda \right) }{\partial \lambda }+k_{22}\frac{\partial
\phi _{2}\left( h^{-1}\left( a\right) ,\lambda \right) }{\partial \lambda }.$
Let $\lambda _{0}$\ be a double eigenvalue, and\ $\phi ^{0}\left( t,\lambda
_{0}\right) $ one of the corresponding eigenfunctions. Then the conditions $%
\Delta \left( \lambda _{0}\right) =0,~\frac{d\Delta \left( \lambda
_{0}\right) }{d\lambda }=0$\ should be fulfilled simultaneously, i.e.,%
\begin{equation}
k_{21}\phi _{1}^{0}\left( a,\lambda _{0}\right) +k_{22}\phi _{2}^{0}\left(
h^{-1}\left( a\right) ,\lambda _{0}\right) =0,  \tag{2.39}
\end{equation}%
\begin{equation}
k_{21}\frac{\partial \phi _{1}^{0}\left( a,\lambda _{0}\right) }{\partial
\lambda }+k_{22}\frac{\partial \phi _{2}^{0}\left( h^{-1}\left( a\right)
,\lambda _{0}\right) }{\partial \lambda }=0.  \tag{2.40}
\end{equation}%
Since $k_{21}~$and $k_{22}$ can not vanish simultaneously, it follows from
that 
\begin{equation}
\phi _{1}^{0}\left( a,\lambda _{0}\right) \frac{\partial \phi _{2}^{0}\left(
h^{-1}\left( a\right) ,\lambda _{0}\right) }{\partial \lambda }-\phi
_{2}^{0}\left( h^{-1}\left( a\right) ,\lambda _{0}\right) \frac{\partial
\phi _{1}^{0}\left( a,\lambda _{0}\right) }{\partial \lambda }=0.  \tag{2.41}
\end{equation}%
Now, differentiating the $q,\omega -$system (1.7) with respect to $\lambda $%
, we obtain%
\begin{equation}
\left\{ 
\begin{array}{l}
-\dfrac{1}{q}D_{\frac{1}{q},\frac{-\omega }{q}}\left( \dfrac{\partial
y_{2}\left( t,\lambda \right) }{\partial \lambda }\right) +\left( p\left(
t\right) -\lambda \right) \dfrac{\partial y_{1}\left( t,\lambda \right) }{%
\partial \lambda }=y_{1}\left( t,\lambda \right) , \\ 
D_{q,\omega }\left( \dfrac{\partial y_{1}\left( t,\lambda \right) }{\partial
\lambda }\right) +\left( r\left( t\right) -\lambda \right) \dfrac{\partial
y_{2}\left( t,\lambda \right) }{\partial \lambda }=y_{2}\left( t,\lambda
\right) .%
\end{array}%
\right.  \tag{2.42}
\end{equation}

Multiplying the $q,\omega -$system (1.7) and (2.42) by $\frac{\partial
y_{1}\left( t,\lambda \right) }{\partial \lambda },~\frac{\partial
y_{2}\left( t,\lambda \right) }{\partial \lambda },~-y_{1}\left( t,\lambda
\right) ~$and $-y_{2}\left( t,\lambda \right) ,$ respectively, adding them
together and applying the $q,\omega -$ integration, we obtain%
\begin{equation}
\begin{array}{l}
\left. \left\{ y_{2}\left( h^{-1}\left( t\right) ,\lambda \right) \dfrac{%
\partial y_{1}\left( t,\lambda \right) }{\partial \lambda }-y_{1}\left(
t,\lambda \right) \dfrac{\partial y_{2}\left( h^{-1}\left( t\right) ,\lambda
\right) }{\partial \lambda }\right\} \right\vert _{\omega _{0}}^{a} \\ 
=\dint\limits_{\omega _{0}}^{a}\left\{ \left( y_{1}\left( t,\lambda \right)
\right) ^{2}+\left( y_{2}\left( t,\lambda \right) \right) ^{2}\right\} 
{\small d}_{q,\omega }{\small t.}%
\end{array}
\tag{2.43}
\end{equation}%
Putting $\lambda =\lambda _{0}$, taking account that $\left. \dfrac{\partial
\phi _{1}^{0}\left( t,\lambda _{0}\right) }{\partial \lambda }\right\vert
_{t=\omega _{0}}=\left. \dfrac{\partial \phi _{2}^{0}\left( t,\lambda
_{0}\right) }{\partial \lambda }\right\vert _{t=\omega _{0}}=0$, and using
the equality (2.41), we obtain%
\begin{equation*}
\dint\limits_{\omega _{0}}^{a}\left\{ \left( \phi _{1}^{0}\left( t,\lambda
_{0}\right) \right) ^{2}+\left( \phi _{2}^{0}\left( t,\lambda _{0}\right)
\right) ^{2}\right\} {\small d}_{q,\omega }{\small t=0.}
\end{equation*}%
Hence $\phi _{1}^{0}\left( t,\lambda _{0}\right) =\phi _{2}^{0}\left(
t,\lambda _{0}\right) =0,$ which is impossible. Consequently $\lambda _{0}$
must be a simple eigenvalue.
\end{proof}

\section{Examples}

\begin{remark}
(see $\left[ 16,17\right] $) The $q,\omega -$ sine and cosine functions
defined by (1.18) and (1.19) have real and simple zeros $\left\{ \pm
g_{n}\right\} _{n=1}^{\infty },~\left\{ \pm j_{n}\right\} _{n=1}^{\infty },$
respectively,%
\begin{equation*}
\begin{array}{l}
g_{n}=\omega _{0}+q^{-n}\left( 1-q\right) ^{-1}\left( 1+O\left( q^{n}\right)
\right) , \\ 
j_{n}=\omega _{0}+q^{-n+1/2}\left( 1-q\right) ^{-1}\left( 1+O\left(
q^{n}\right) \right) ,%
\end{array}%
~n\geq 1.
\end{equation*}
\end{remark}

\begin{example}
Consider the $q,\omega -$ Dirac system (1.7)-(1.9) in which $p\left(
t\right) =r\left( t\right) =0:$ 
\begin{equation}
\left\{ 
\begin{array}{l}
-\dfrac{1}{q}D_{\frac{1}{q},\frac{-\omega }{q}}y_{2}=\lambda y_{1}, \\ 
D_{q,\omega }y_{1}=\lambda y_{2},%
\end{array}%
\right.   \tag{3.1}
\end{equation}%
\begin{equation}
y_{1}\left( \omega _{0}\right) =0,  \tag{3.2}
\end{equation}%
\begin{equation}
y_{2}\left( h^{-1}\left( \pi \right) \right) =0.  \tag{3.3}
\end{equation}%
It is easy to see that a solution (3.1) and (3.2) is given by%
\begin{equation*}
\phi ^{\intercal }\left( t,\lambda \right) =\left( -S_{q,\omega }\left(
t,\lambda \right) ,~-C_{q,\omega }\left( t,\sqrt{q}\lambda \right) \right) .
\end{equation*}%
By substituting this solution in (3.3), we obtain $\Delta \left( \lambda
\right) =C_{q,\omega }\left( h^{-1}\left( \pi \right) ,\sqrt{q}\lambda
\right) .$ By the previous Remark 3.1, the eigenvalues are 
\begin{equation}
\lambda _{n}=\dfrac{q^{-n+1}}{\left( 1-q\right) \left( \pi -\omega
_{0}\right) }\left( 1+O\left( q^{n}\right) \right) ,~n=1,2,...~.  \tag{3.4}
\end{equation}%
This approximate the eigenvalues of the $q-$Dirac system as $\omega
\rightarrow 0^{+},$ see Example 1 in $\left[ 20\right] .$
\end{example}

\begin{example}
Consider the $q,\omega -$ Dirac system (3.1) together with the following
boundary conditions 
\begin{equation}
y_{2}\left( \omega _{0}\right) =0,  \tag{3.5}
\end{equation}%
\begin{equation}
y_{2}\left( h^{-1}\left( \pi \right) \right) =0.  \tag{3.6}
\end{equation}%
In this case $\phi ^{\intercal }\left( t,\lambda \right) =\left( C_{q,\omega
}\left( t,\lambda \right) ,~-\sqrt{q}S_{q,\omega }\left( t,\sqrt{q}\lambda
\right) \right) .$\ Since $\Delta \left( \lambda \right) =\sqrt{q}%
S_{q,\omega }\left( h^{-1}\left( \pi \right) ,\sqrt{q}\lambda \right) ,~$%
then the eigenvalues are given by 
\begin{equation}
\lambda _{n}=\dfrac{q^{-n+1/2}}{\left( 1-q\right) \left( \pi -\omega
_{0}\right) }\left( 1+O\left( q^{n}\right) \right) ,~n=1,2,...~.  \tag{3.7}
\end{equation}
\end{example}

\end{document}